\documentclass[12pt]{article}
\usepackage[T2A]{fontenc}
\usepackage[english]{babel}
\usepackage[tbtags]{amsmath}
\usepackage{amsfonts,amssymb}
\usepackage{amssymb}
\usepackage{graphicx}
\usepackage{amscd}

\usepackage{fullpage}
\usepackage{float}

\usepackage{graphics,amsmath,amssymb}
\usepackage{amsthm}
\usepackage{amsfonts}
\usepackage{mathrsfs}

\title{On a property of superposition of the generating functions $\ln\left(\frac{1}{1-F(x)}\right)$}

\author{Dmitry Kruchinin\\
\small Tomsk State University of Control Systems and Radioelectronics, Russian Federation\\
\small \texttt{kruchininDm@gmail.com}\\
}

\begin{document}
\maketitle

\begin{abstract}
Obtained a new property of superposition of the generating functions $ \ln\left(\frac {1} {1-F(x)}\right)$, where $ F (x) $ -- generating function with integer coefficients, which allows the construction a primality tests.
The theorem which is based on compositions of positive numbers and its corollary are proved. Examples are given.

Key words: Generating functions, superposition of generating functions, composition of a natural number.
\end{abstract}

\theoremstyle{plain}
\newtheorem{theorem}{Theorem}
\newtheorem{corollary}[theorem]{Corollary}
\newtheorem{lemma}[theorem]{Lemma}
\newtheorem{proposition}[theorem]{Proposition}

\theoremstyle{definition}
\newtheorem{definition}[theorem]{Definition}
\newtheorem{example}[theorem]{Example}
\newtheorem{conjecture}[theorem]{Conjecture}
\theoremstyle{remark}
\newtheorem{remark}[theorem]{Remark}

\newtheorem{Theorem}{Theorem}[section]
\newtheorem{Proposition}[Theorem]{Proposition}
\newtheorem{Corollary}[Theorem]{Corollary}

\theoremstyle{definition}
\newtheorem{Example}[Theorem]{Example}
\newtheorem{Remark}[Theorem]{Remark}
\newtheorem{Problem}[Theorem]{Problem}
\newtheorem{state}[Theorem]{Statement}
\makeatletter
\def\rdots{\mathinner{\mkern1mu\raise\p@
\vbox{\kern7\p@\hbox{.}}\mkern2mu
\raise4\p@\hbox{.}\mkern2mu\raise7\p@\hbox{.}\mkern1mu}}
\makeatother

Generating functions are a powerful tool for solving problems in number theory, combinatorics, algebra and probability theory. One of the directions in the theory of generating functions is the investigation of the coefficients of powers of generating functions, which play an essential role in the operation of superposition of generating functions. In \ cite {bibBook1} shows the possible ways of calculating these coefficients. This paper continues the study of the coefficients of superposition, provided that the external function is the generating function of the logarithm, and the internal generating function with integer coefficients.
   
Consider the following generating functions $F(x)=\sum_{n\geq1}f(n)x^{n}$ and $R(x)=\sum_{n\geq0}r(n)x^{n}$, $ f(n), r (n) $ - integer-valued functions and formulate the following theorem.

\begin{theorem}
\label{1}
Let there be given the generating function $ F(x)=\sum_{n\geq1} f (n) x^{n}$ and $ R (x) = \sum_{n\geq0} r(n)x^{n}$, where $ f (n), r (n) $ - integer-valued functions. Then, for the superposition $$ G \left (x \right) = R \left (F \left (x \right) \right) = \sum_ {n \geq0} g (n) x ^ {n} $$ 
$g (n) $ is an integral function.
\end{theorem}
\begin{proof}	
Consider the formula for calculation $g(n)$ of  the superposition of the generating functions
\cite{bibBook1}
	$$g(0)=r(0)$$
		\begin {equation}
 		\label {comp1}
 		g(n)=\sum^{n}_{k=1}\sum_{{\lambda_i>0 \atop 		\lambda_1+\lambda_2+\ldots+\lambda_k=n}}f(\lambda_{1})f(\lambda_{2})\ldots f(\lambda_{k})r(k)=\sum^{n}_{k=1}F^{\Delta}(n,k)r(k),
 		\end {equation}
where $F^{\Delta}(n,k)=\sum_{{\lambda_i>0 \atop 		\lambda_1+\lambda_2+\ldots+\lambda_k=n}}f(\lambda_{1})f(\lambda_{2})\ldots f(\lambda_{k})$ is the compositae of generating function $F(x)=\sum_{n\geq1}f(n)x^{n}$ \cite{bibBook1}.

Because $f(n)$ is an integral function, the values of compositae $F^{\Delta}(n,k)$ are integers.Consequently, the coefficients of the superposition $G\left(x\right)=R\left(F\left(x\right)\right)$ are also integers.
\end{proof}	

For further arguments prove the following theorem.
	
\begin{theorem} 
	The sum
		\begin{equation}
		\label{sum_1}
			\sum_{k=1}^n\frac{n}{k}\sum_{{\lambda_i>0 \atop 		\lambda_1+\lambda_2+\ldots+\lambda_k=n}} a_{\lambda_1}a_{\lambda_2}\ldots a_{\lambda_k} 
		\end{equation} 
		is integer for any integer sequence $a_1, a_2, \ldots, a_n$.
\end{theorem}

\begin{proof}
Let us construct a generating function $F(x)=a_1x+a_2x^2+\ldots+a_nx^n+\ldots$. Then 
$$
F^{\Delta}(n,k)=\sum_{{\lambda_i>0 \atop \lambda_1+\lambda_2+\ldots+\lambda_k=n}} a_{\lambda_1}a_{\lambda_2}\ldots a_{\lambda_k}.
$$
is the compositae of this generating function according to its definition.

Hence coefficients of superposition of generating functions
$G(x)=\ln\left(\frac{1}{1-F(x)}\right)$ are given by
$$
g(n)=\sum_{k=1}^n \frac{F^{\Delta}(n,k)}{k},
$$
$$
G(x)=\sum_{n>0} g(n)x^n.
$$
If consider derivative $G'(x)=\left[\ln\left(\frac{1}{1-F(x)}\right)\right]'$
we can obtain the following expression
$$
\left(\frac{F'(x)}{1-F(x)}\right)=g_1+2g_2x^1+\ldots+ng_nx^{n-1}+\ldots
$$
Consider the left part as product of generating functions $F'(x)$ and $\left(\frac{1}{1-F(x)}\right)$. Coefficients of $F'(x)$ are integers. According to Theorem 1, coefficients of superposition of generating functions $H(x)=\left(\frac{1}{1-F(x)}\right)$ are also integers, because $F(x)$ и $R(x)=\left(\frac{1}{1-x}\right)$ are generating functions with integer coefficients.

Product of functions with integer coefficients also have integer coefficients. Hence the expression for the coefficients
\begin{equation}
\label{sum3}
n\sum_{k=1}^n \frac{F^{\Delta}(n,k)}{k}
\end{equation}   
are integers.

The theorem is proved.
\end{proof}

Consider some simple examples.
\begin{Example}
Let $a_n$ be prime integers and  $a_1=1$. Then for  $n=6$ we have:
\begin{multline*}
n\sum_{k=1}^n\frac{1}{k}\sum_{{\lambda_i>0 \atop 		\lambda_1+\lambda_2+\ldots+\lambda_k=n}} a_{\lambda_1}a_{\lambda_2}\ldots a_{\lambda_k}=6\left(a_{6}+\frac{a_{1}a_{5}+a_{2}a_{4}+a_{3}a_{3}+a_{5}a_{1}}{2}+\ldots+\frac{a_{1}a_{1}a_{1}a_{1}a_{1}a_{1}}{6}\right)=\\
=6\left(11+\frac{14+20+9}{2}+\frac{15+36+8}{3}+\frac{12+24}{4}+\frac{10}{5}+\frac{1}{6}\right)=380
\end {multline*}
\end{Example}
\begin{Example}
$a_i=1$, $i=\overline{1,n}$ then
$$
\sum_{{\lambda_i>0 \atop \lambda_1+\lambda_2+\ldots+\lambda_k=n}} a_{\lambda_1}a_{\lambda_2}\ldots a_{\lambda_k} ={n-1 \choose k-1}. 
$$
because it accounts the number of $n$ compositions that have $k$ parts, $n$-- positive number. Hence
$$
\sum_{k=1}^n \frac{n}{k} {n-1 \choose k-1}=2^{n}-1.
$$
\end{Example}

\begin {corollary}
For any integer sequence $a_1, a_2, \ldots, a_n$, where $n$ is a prime number, sum 
\begin{equation}
\label{sum_2}
\sum_{k=1}^{n-1}\frac{1}{k}\sum_{{\lambda_i>0 \atop \lambda_1+\lambda_2+\ldots+\lambda_k=n}} a_{\lambda_1}a_{\lambda_2}\ldots a_{\lambda_k} 
\end{equation} 
is integer.
\end {corollary}

\begin{proof}
Consider the expression (\ref{sum3}).
When $k=n$, the expression (\ref{sum3}) is equal to the integer value $a^{n}_{1}$ \cite{bibBook1}. Hence the expression
$$
n\sum_{k=1}^{n-1} \frac{F^{\Delta}(n,k)}{k}
$$
 is an integer.
$n$ have not common divisors of $k<n$, because $n$ is the prime.
Hence the expression (\ref{sum_2}) is an integer.
\end{proof}

Hence the value of expression 
\begin{equation}
\label{sum_3}
\sum_{k=1}^{n-1}\frac{F^{\Delta}(n,k)}{k}
\end{equation}
is integer for any $n$, that are prime numbers. The converse is false, i.e. if $n$ is not prime, then proper value (\ref{sum_3}) may be either integer or not.                                            

Consider some specific examples.

\begin{Example}
Refer to the example above.
$$
\sum_{k=1}^n \frac{n}{k} {n-1 \choose k-1}=2^{n}-1.
$$

Hence the value of expression 
$$
\sum_{k=1}^{n-1} \frac{1}{k} {n-1 \choose k-1}=\frac{2^{n}-2}{n}
$$
is integer for prime $n$.
\end{Example}

\begin{Example}
Let us have a generating function  $F(x)=x+x^2$ and its compositae,according to the source \cite{bibBook1},$F^{\Delta}(n,k)={k \choose n-k}$, then coefficients of superposition $\ln\left(\frac{1}{1-x-x^2}\right)$ are given by
$$
g_n=\sum_{k=1}^n {k \choose n-k}\frac{1}{k},
$$
$$
ng_n=[1,3,4,7,11,18,29,47,76,123,199,322,521,843,1364,2207,3571].
$$
This formula generates Lukas numbers (А000032)\cite{bibUrl}. Hence the value of expression 
$$
\frac{L_n-1}{n}
$$
is integer for prime numbers, where $L_n$ - Lukas numbers.
$$
L(n)=\left(\frac{1+\sqrt{5}}{2}\right)^n + \left(\frac{1-\sqrt{5}}{2}\right)^n
$$
or
$$
L(n)=Fib(n)+2Fib(n-1)=Fib(n + 1)+Fib(n-1).
$$
\end{Example}

\begin{Example}
Let us have the generating function for Catalan numbers  $F(x)=\frac{1-\sqrt{1-4x}}{2x}$ and its compositae, due to the source \cite{bibBook1}, 
$F^{\Delta}(n,k)=\frac{k}{n}{2n-k-1\choose n-1}$, then coefficients of superposition $\ln\left(\frac{1}{1-F(x))}\right)$ are given by
$$
g_n=\sum_{k=1}^n \frac{k}{n}{2n-k-1\choose n-1}\frac{1}{k}=\frac{1}{n}\sum_{k=1}^n{2n-k-1\choose n-1},
$$
$$
ng_n=[1,3,10,35,126,462,1716,6435,24310,92378].
$$
This formula induces sequence of integers А001700\cite{bibUrl}, wherefrom
$$
ng_n={2n-1 \choose n-1}.
$$
Hence value of expression
$$
\frac{1}{n}\left({2n-1 \choose n-1}-1\right)
$$
is integer for prime numbers.
\end{Example}

This result allows us to construct algorithms which are based on superposition of the generating functions $\ln\left(\frac{1}{1-F(x)}\right)$, where $F(x)$ is generating function with integer coefficients,  for verification of the positive numbers' primality.

\end{document}